\documentclass[a4paper,11pt]{amsart}
\usepackage{extarrows}
\usepackage{amssymb}
\usepackage{mathrsfs}
\usepackage{amsmath,amssymb,amsthm,latexsym,amscd,mathrsfs}
\usepackage{indentfirst}
\usepackage{extarrows}

\usepackage{multirow}

\usepackage{latexsym}
\usepackage{amssymb}
\usepackage{amsmath}
\usepackage{amsthm}
\usepackage{amsfonts}
\usepackage{color}
\usepackage{pictexwd,dcpic}
\usepackage{graphicx}
\usepackage{psfrag}
\usepackage{hyperref}
\usepackage{comment}

\numberwithin{equation}{section} \theoremstyle{plain}
\newtheorem{thm}{Theorem}[section]

\newtheorem{prop}[thm]{Proposition}
\newtheorem{lem}[thm]{Lemma}
\newtheorem{cor}[thm]{Corollary}

\newtheorem{conj}[thm]{Conjecture}

\newtheorem{rem}[thm]{Remark}

\newtheorem{ack}{Acknowledgements}   
  \makeatletter

\newcommand{\fol}{\mathcal{F}}

\newcommand{\In}{\subseteq}
\newcommand{\lra}{\longrightarrow}
\newcommand{\sphere}{\mathbb{S}}

\newcommand{\scal}[1]{\langle #1 \rangle}

\def\P{\ensuremath{\mathsf{p}}}
\DeclareMathOperator{\Diff}{Diff}

\DeclareMathOperator{\Iso}{Isom}
\DeclareMathOperator{\dist}{dist}

\DeclareMathOperator{\img}{Image}

\newcommand{\ZZ}{\mathbb{Z}}

\def\ol{\overline}

\def\lra{\longrightarrow}
\def\lmt{\longmapsto}

\def\In{\subseteq}

\def\CC{\mathbb{C}}
\def\DD{\mathbb{D}}

\def\PP{\Bbb{P}}
\def\QQ{\Bbb{Q}}
\def\RR{\mathbb{R}}
\def\ZZ{\mathbb{Z}}

\def\mc{\mathcal}

\def\fol{\mc{F}}

\def\O{\textrm{O}}

\makeatother
\title[Differentiable classification in dimension 4]{Differentiable classification of 4-manifolds with singular Riemannian foliations}
\author[J.Q. Ge]{Jianquan Ge}
\address{School of Mathematical Sciences, Laboratory of Mathematics and Complex Systems, Beijing Normal
University, Beijing 100875, P.R. CHINA.}
\email{jqge@bnu.edu.cn}
\author[M. Radeschi]{Marco Radeschi}
\address{Mathematisches Institut, WWU M\"{u}nster, Einsteinstr. 62, M\"{u}nster, Germany.}
\email{mrade\_02@uni-muenster.de}
\date{}
\thanks{The first author is partially supported by the NSFC (No. 11001016 and No. 11331002), the SRFDP (No. 20100003120003), the Fundamental Research
Funds for the Central Universities, and a research fellowship from the Alexander von Humboldt Foundation. }
\begin{document}
\maketitle

\begin{abstract}
In this paper, we first prove that any closed simply connected 4-manifold that admits a decomposition into two disk bundles of rank greater than 1 is diffeomorphic to one of the standard elliptic 4-manifolds: $\mathbb{S}^4$, $\mathbb{CP}^2$, $\mathbb{S}^2\times\mathbb{S}^2$, or
$\mathbb{CP}^2\#\pm \mathbb{CP}^2$. As an application we prove that any closed simply connected 4-manifold admitting a nontrivial singular Riemannian foliation is diffeomorphic to a connected sum of copies of standard $\mathbb{S}^4$, $\pm\mathbb{CP}^2$ and $\mathbb{S}^2\times\mathbb{S}^2$. A classification of singular Riemannian foliations of codimension 1 on all closed simply connected 4-manifolds is obtained as a byproduct. In particular, there are exactly 3 non-homogeneous singular Riemannian foliations of codimension 1, complementing the list of cohomogeneity one 4-manifolds.
\end{abstract}

\section{Introduction}\label{introduction}

Four-dimensional manifolds form an extremely rich and interesting class of manifolds\footnote{Throughout this paper, all manifolds considered are connected and smooth.}. This is the lowest dimension in which exotic smooth structures arise, e.g., the noncompact $4$-space $\mathbb{R}^4$  \cite{Do83, Fr82, Go83,  Tau87} and compact $m\mathbb{CP}^2\# n\overline{\mathbb{CP}^2}$ for many pairs of $(m\geq1,n\geq2)$ \cite{AP2, Do87, FS05}. Moreover, exotic smooth structures abound in this dimension and it is not known whether there is a $4$-manifold with only one (standard) smooth structure, even for the simplest $4$-manifold $\sphere^4$, the affirmative side of which is called the \emph{smooth Poincar\'{e} conjecture in dimension $4$}.

It is thus natural to try to classify subclasses of 4-manifolds with additional structures. The study of 4-manifolds admitting smooth group actions received a lot of attention, and there is rich literature on the subject. Joining several different independent results (cf.  \cite{Fi77, Fi78, OR70, OR74, Pa86}) we know that any closed simply connected 4-manifold admitting a smooth action by a compact Lie group is diffeomorphic to a connected sum of copies of standard $\sphere^4$, $\pm\CC\PP^2$ and $\sphere^2\times\sphere^2$.
When admitting a cohomogeneity one action, the closed simply connected $4$-manifold splits as a union of two disk bundles, glued along their common boundary. In this case the classification was carried out by Parker \cite{Pa86} (see also \cite{GH87, GZ11, Ho10}), who proved that such manifolds must be diffeomorphic to $\mathbb{S}^4$, $\mathbb{CP}^2$, $\mathbb{S}^2\times\mathbb{S}^2$, or $\mathbb{CP}^2\#- \mathbb{CP}^2$.

The first part of this paper is concerned with a classification of closed simply connected $4$-manifolds which admit a splitting structure into disk bundles but without requiring any group action.
\begin{thm}\label{splitting thm}
Let $N$ be a closed   simply connected $4$-manifold obtained by gluing two disk bundles over closed submanifolds of codimension greater than $1$. Then
$N$ is diffeomorphic to one of the standard $\mathbb{S}^4$, $\mathbb{CP}^2$, $\mathbb{S}^2\times\mathbb{S}^2$, or
$\mathbb{CP}^2\#\pm \mathbb{CP}^2$.
\end{thm}

In \cite{GT12}, the first named author and Tang proved that if a homotopy $4$-sphere admits some properly transnormal function, or equivalently it has a splitting structure as in Theorem \ref{splitting thm}, then it must be diffeomorphic to the standard $\mathbb{S}^4$. Thus, Theorem \ref{splitting thm} was known when $N$ is a homotopy (topological) $4$-sphere. As an immediate application, we see that there exist no properly transnormal (isoparametric) function on any closed simply connected $4$-manifold other than the five standard elliptic $4$-manifolds. This should be compared to the interesting existence result of Qian and Tang \cite{QT13} that every homotopy $n$-sphere $(n>4)$ carries a properly isoparametric function.

 Recall that a singular Riemannian foliation on a Riemannian manifold is, roughly speaking, a partition of $M$ into connected complete, injectively immersed submanifolds which stay at a constant distance from each other, and it provides a generalization of smooth actions of Lie groups. If the foliation has codimension 1, it gives rise to a splitting structure as in Theorem \ref{splitting thm}. Therefore, Theorem \ref{splitting thm} provides a generalization of Parker's result on cohomogeneity one 4-manifolds. Moreover,
 we obtain
 \begin{cor}
 There are exactly 3 foliated diffeomorphism classes of non-homogeneous singular Riemannian foliations of codimension 1 in closed simply connected 4-manifolds.
 \end{cor}
  In fact, in Corollary \ref{SRFcodim1} we will recover all cohomogeneity one actions, together with these three non-homogeneous singular Riemannian foliations of codimension 1, one of which on $\mathbb{CP}^2\#-\mathbb{CP}^2$ and the other two on $\mathbb{CP}^2\#\mathbb{CP}^2$ (see some further description in Subsection \ref{3-non-homog}). Notice that $\mathbb{CP}^2\#\mathbb{CP}^2$ does not admit cohomogeneity one actions, though it indeed admits metrics of non-negative curvatures as the other four cohomogeneity one 4-manifolds. This suggests the following
\begin{conj}
Any closed simply connected non-negatively curved Riemannian manifold admits singular Riemannian foliations of codimension 1, under the same metric or a different bundle-like metric.
\end{conj}
 Existence of such foliations ensures that topologically the manifold admits a splitting structure into two disk bundles as in Theorem \ref{splitting thm}. Hence this conjecture is essentially attributed to Karsten Grove who conjectured that every compact non-negatively curved manifold splits as a union of two disk bundles. A positive answer to this conjecture will in particular solve affirmatively the long-standing conjecture: a closed simply connected non-negatively curved 4-manifold is diffeomorphic to one of the standard $\mathbb{S}^4$, $\mathbb{CP}^2$, $\mathbb{S}^2\times\mathbb{S}^2$, or $\mathbb{CP}^2\#\pm \mathbb{CP}^2$. The latter conjecture is still open even for homeomorphism.

In the second part of this paper we turn to consider singular Riemannian foliations of general codimension. At last we are able to recover and generalize the differentiable classification obtained for group actions.
\begin{thm}\label{SRF thm}
Let $N$ be a closed simply connected $4$-manifold admitting a nontrivial singular Riemannian foliation. Then it is diffeomorphic to a connected sum of
copies of standard $\mathbb{S}^4$, $\pm \mathbb{CP}^2$ and $\mathbb{S}^2\times \mathbb{S}^2$.
\end{thm}

When the singular Riemannian foliation is closed and of dimension $1$, this result has been proven by Galaz-Garcia and the second named author \cite[Cor. 8.6]{GaRa13}, by showing that such a foliation comes from a smooth effective circle action which then derives the conclusion from the classification of circle actions on $4$-manifolds by Fintushel \cite{Fi77, Fi78}. When the foliation is of dimension $3$, Theorem \ref{SRF thm} reduces to Theorem \ref{splitting thm}.

\section{Preliminaries}\label{Preliminaries}
In this section we collect some background materials on singular Riemannian foliations, most part of which is based on the Preliminary section of \cite{GaRa13}. We also refer the reader to \cite{Mol88,ABT13} for further results on the theory.
\subsection{Singular Riemannian foliations}
\label{SS:SRF} A \emph{transnormal system} $\fol$ on a complete Riemannian manifold $M$ is a decomposition of $M$ into complete, injectively
immersed connected submanifolds, called \emph{leaves}, such that every geodesic emanating perpendicularly to one leaf remains perpendicular to all
leaves. A \emph{singular Riemannian foliation}  is a transnormal system $\mathcal{F}$ which is also a \emph{singular foliation}, i.e., such that there are smooth vector fields $X_i$ on $M$ that span the tangent space $T_p L_p$ to the leaf $L_p$ through each point $p\in M$. If furthermore $\mathcal{F}$ is regular, i.e., the leaves have the same dimension, then $\mathcal{F}$ is called a \emph{Riemannian foliation}.
If $M$ is a smooth manifold and $(M, \fol)$ is a singular foliation, a metric $g$ on $M$ is called \emph{bundle-like} for $\fol$ if $(M,g,\fol)$ becomes a singular Riemannian foliation. Slightly abusing notation, the
pair $(M,\fol)$ will also denote a singular Riemannian foliation $\fol$ on a complete Riemannian manifold $(M,g)$.

We will call the quotient space $M/\fol$ the \emph{leaf space}, and will also denote it by $M^*$. We will let $\pi: M\to M/\fol$ be the leaf projection map. A singular Riemannian foliation $\fol$ will be called \emph{closed} if all its leaves are closed in $M$.
If $\fol$ is closed, then the leaf space $M/\fol$ is a Hausdorff metric space.

A leaf of maximal dimension is called a \emph{regular} leaf, and its dimension is defined to be the \emph{dimension} of $\mathcal{F}$, denoted by $\dim \mathcal{F}$. Leaves of lower dimensions are called \emph{singular} leaves. The \emph{codimension} of $\mathcal{F}$ is defined to be the codimension of a regular leaf. A singular Riemannian foliation $(M,\fol)$ of codimension one is called an \emph{isoparametric} foliation if the regular leaves have constant mean curvature, and a \emph{totally isoparametric} foliation if the regular leaves have constant principal curvatures in $M$. In fact, when $(M,\fol)$ gives a splitting structure as (\ref{splitting str}), it can become isoparametric by modifying the bundle-like metric on $M$ (cf. \cite{QT13}). A codimension one closed singular Riemannian (resp. isoparametric, totally isoparametric) foliation on a closed Riemannian manifold would be given by level sets of a transnormal (resp. isoparametric, totally isoparametric) function, which we do not introduce here but refer the reader to \cite{GT12, GTY11, GW09, QT13, Th10, Wa87}.

\subsection{Stratification}
Let $(M,\fol)$  be a singular Riemannian foliation. For any point $p\in M$, we denote by $L_p$ the leaf of $\fol$ through $p$.
For $k\leq \dim \mathcal{F}$, define
\[
\Sigma_{k} =\{\, p\in M : \dim L_p =k\,\}.
\]
Every connected component $C$ of the set $\Sigma_{k}$, called a \emph{stratum}, is an embedded (possibly non-complete) submanifold of $M$ and the restriction of $\fol$ to $C$
is a Riemannian foliation. Moreover, any horizontal geodesic (perpendicular to the leaves) tangent to $\Sigma_k$, stays in the closure of $\Sigma_k$ for all time.
The subset $\Sigma_{\dim \fol}$ of regular leaves is open, dense and connected in $M$; it is called the \emph{regular stratum} of $M$, and it will be  denoted by $M_0$. All other strata have codimension
at least $2$ in $M$ and are called \emph{singular strata}.

The quotient $M/\fol$ inherits a stratification from $M$, where the strata are the projections $\Sigma/\fol$ of the strata $\Sigma$ of $\fol$.  Any such stratum $\Sigma/\fol$ is an orbifold and in particular the regular stratum $M_0/\fol$ is an orbifold which is open and dense in $M/\fol$.

\subsection{Holonomy map}\label{SS:holonomy}The fundamental group $\pi_1(L_p)$ acts on the normal space $\nu_pL_p$ of a regular leaf $L_p$, in such a way that if $x,y\in \nu_pL_p$ belong to the same $\pi_1(L_p)$-orbit, then $\exp_p(tx)$, $\exp_p(ty)$ belong to the same leaf, for all $t$. Such an action is called \emph{holonomy map}. Fixed an $\epsilon$-tubular neighborhood $U$ of $L_p$ for some small $\epsilon$, the universal cover $\tilde{U}$ admits a foliation $\tilde{\fol}$ which is the lift of $(U,\fol)$. One checks that $\tilde{U}$ splits as a product $\tilde{L}_{p}\times D$, where $D$ is an $\epsilon$-disk in $\nu_pL_p$ around the origin. Therefore $U=\tilde{L}_p\times_{\pi_1(L_p)}D$, where $\pi_1(L_p)$ acts by deck transformations on the first factor, and by the holonomy map on the second. In particular, the normal bundle is orientable if and only if the holonomy map acts preserving the orientation of $\nu_pL_p$. Moreover, the holonomy group is the local group of the orbifold $M_0/\Sigma$ at the point $\pi(p)$ (cf. \cite[Section 3.6]{Mol88}).

A regular leaf $L$ is called \emph{principal} if the holonomy map acts trivially on the normal space $\nu_pL$ (the definition is independent on the point $p\in L$) and \emph{exceptional} otherwise. From the local description of $\fol$ around $L$, a regular leaf is principal if and only if it projects to a manifold point of $M_0/\fol$.

\subsection{Infinitesimal singular Riemannian foliations}
\label{SSS:INF_SRF} Given a point $p\in M$, let $\sphere_p$ be
the unit sphere in the normal space $\nu_pL_p$ of the leaf through $p$. On $\sphere_p$ we define a foliation $\fol_p$ by saying that $x,y\in
\sphere_p$ belong to the same leaf in $\fol_p$ if $\exp_p(\epsilon x)$ and $\exp_p(\epsilon y)$ belong to the same leaf of $\fol$, for every
$\epsilon>0$ small enough. If $\sphere_p$ is equipped with the round metric, the foliation $(\sphere_p,\fol_p)$ is a singular Riemannian foliation,
and it is called the \emph{infinitesimal foliation} of $\fol$ at $p$. If $p$ is a regular point, $\fol_p$ is the trivial foliation whose leaves are points.

Infinitesimal foliations are useful to understand the relation between leaves of different types, as follows. Consider a point $p\in M$, a vector $x\in \sphere_p$, and let $q=\exp_p\epsilon x$. If $\epsilon$ is small enough, there is a well defined, smooth closest-point projection $\mathsf{p}:L_q\to L_p$, that is a locally trivial fibration. Moreover, the connected component of the fiber of $\mathsf{p}$ through $q$ can be identified with the leaf $\mathcal{L}_x\in \fol_p$ through $x$. There is a cover $\ol{L}_p\to L_p$ of $L_p$ such that $\mathsf{p}$ lifts to a fibration $\ol{\mathsf{p}}:L_q\to \overline{L}_p$ with connected fiber $\mathcal{L}_x$ (just take $\ol{L}_p=\tilde{L}_p/\mathsf{p}_*(\pi_1(L_q))$ where $\tilde{L}_p$ is the universal cover of $L_p$):
\begin{equation}
\label{E:Inf_fibration}
\mathcal{L}_x\lra L_q\lra \ol{L}_p.
\end{equation}

\section{Singular Riemannian foliations of codimension one}\label{Proof of splitting thm}
In this section we first prove Theorem \ref{splitting thm} which is essentially the codimension $1$ case of Theorem \ref{SRF thm}, and then give a classification of singular Riemannian foliations of codimension $1$ on closed simply connected $4$-manifolds in Corollary \ref{SRFcodim1}. Throughout this paper we denote diffeomorphisms
and homeomorphisms by $``\cong"$ and $``\simeq"$ respectively.

Let $N$ be a closed  simply connected $4$-manifold obtained by gluing two disk bundles over closed submanifolds $M_{\pm}$ of codimension
$m_{\pm}$ greater than one, i.e.,
\begin{equation}\label{splitting str}
N=D(M_{+})\bigcup_{f}D(M_{-}),
\end{equation}
where $f: \partial D(M_{+})\rightarrow\partial D(M_{-})$ is a diffeomorphism between the boundaries of the disk bundles $D(M_{\pm})$ of rank $m_{\pm}$
over $M_{\pm}$. We denote the common boundary by $M:=\partial D(M_{+})\cong\partial D(M_{-})$ and it follows from the proof of \cite[Cor. 11.4 and
Thm. 11.3]{MS74} that $M$ is an orientable hypersurface of $N$. Without loss of generality, we assume $2\leq m_{+}\leq
m_{-}\leq4$. As remarked in the introduction, the case when $N$ is a $4$-sphere has been solved in \cite{GT12} and henceforth we assume $\dim H^2(N,\mathbb{Z}_2)=b_2\geq1$ for simplicity, although this case could also be derived by the same arguments.

For a splitting structure as (\ref{splitting str}), we have the following exact cohomology sequences (\cite{Du11,He81}):
\begin{equation}\label{exact cohom seq}
\cdots\rightarrow H^{i-1}(M_{\pm},\mathbb{Z}_2)\xlongrightarrow{\theta_{\pm}} H^{i-m_{\mp}}(M_{\mp},\mathbb{Z}_2)\xlongrightarrow{\alpha_{\pm}} H^i(N,
\mathbb{Z}_2)\xlongrightarrow{i_{\pm}^*} H^{i}(M_{\pm},\mathbb{Z}_2)\rightarrow\cdots.
\end{equation}
Since $N$ is simply connected, by Poincar\'{e} duality we have $H^1(N,\mathbb{Z}_2)=H^3(N,\mathbb{Z}_2)=0$.

To prove Theorem \ref{splitting thm} we analyze case by case according to the value of $m_{-}\in\{2,3,4\}$ in the following subsections.
\subsection{$m_{-}=4$}
Then $M_{-}=\{pt\}$ is a point in $N$ and $M\cong\mathbb{S}^3$. By (\ref{exact cohom seq}) we have
\begin{equation*}
0\longrightarrow H^{2-m_{+}}(M_{+},\mathbb{Z}_2)\xlongrightarrow{\alpha_{-}}H^2(N, \mathbb{Z}_2)\xlongrightarrow{i_{-}^*}
0\xlongrightarrow{\theta_{-}}H^{3-m_{+}}(M_{+},\mathbb{Z}_2)\longrightarrow 0,
\end{equation*}
which, due to the assumption $b_2\geq1$, implies $m_{+}=2$, $H^1(M_{+},\mathbb{Z}_2)=0$ and $H^2(N,
\mathbb{Z}_2)=H^0(M_{+},\mathbb{Z}_2)=\mathbb{Z}_2$. Because $2$-dimensional manifolds are determined by their cohomology structures and closed   simply
connected $4$-manifolds are determined up to homeomorphism by the second Betti number when it is less than $2$ (cf. \cite{FQ90}) , the equalities above lead to
$M_{+}\cong\mathbb{S}^2\cong\mathbb{CP}^1$, $N\simeq\mathbb{CP}^2$ and
\begin{equation*}
N=D(\mathbb{CP}^1)\bigcup_f D^4,
\end{equation*}
where $f\in \Diff(\mathbb{S}^3)$ is a diffeomorphism of the common boundary $\mathbb{S}^3$ of the $4$-disk $D^4$ and the $2$-disk bundle
$D(\mathbb{CP}^1)$. Meanwhile it is known that one of the two isoparametric foliations in $\mathbb{CP}^2$ (Both are homogeneous! cf. \cite{GTY11},
etc.) splits $\mathbb{CP}^2$ as
\begin{equation*}
\mathbb{CP}^2=D(\mathbb{CP}^1)\bigcup_{id} D^4,
\end{equation*}
where $\mathbb{CP}^1\subset \mathbb{CP}^2$ is the canonical inclusion and $id$ is the identity map of the distance sphere around the focal point of
$\mathbb{CP}^1$ in $\mathbb{CP}^2$. Note that one still gets $\mathbb{CP}^2$ if the gluing map $id$ were replaced by an orientation reversing isometry
(e.g. a reflection) since it can be radially extended to the $4$-disk $D^4$.  Now by Cerf \cite{Ce68} and Hatcher \cite{Ha83}, $f\in \Diff(\mathbb{S}^3)$ is
isotopic to $id$ up to an orientation reversing isometry, and therefore using standard argument with the isotopy extension theorem (cf. \cite[Thm. 2.3]{Hi76}) we obtain $N\cong\mathbb{CP}^2$.
\subsection{$m_{-}=3$}
Then $M_{-}\cong \mathbb{S}^1$ and $M\cong \mathbb{S}^1\times \mathbb{S}^2$ since the only nontrivial $2$-sphere bundle over $\mathbb{S}^1$ is
non-orientable and thus impossible as discussed before. Using (\ref{exact cohom seq}) again we get the short exact sequence
\begin{equation*}
0\longrightarrow\mathbb{Z}_2 \xlongrightarrow{\theta_{-}}H^{2-m_{+}}(M_{+},\mathbb{Z}_2)\xlongrightarrow{\alpha_{-}}H^2(N,
\mathbb{Z}_2)\longrightarrow 0,
\end{equation*}
which implies $m_{+}=2$ and $\mathbb{Z}_2=H^0(M_{+},\mathbb{Z}_2)=\mathbb{Z}_2\oplus H^2(N, \mathbb{Z}_2)$, contradicting with the assumption
$b_2\geq1$. This case occurs only if $N\cong \sphere^4$.
\subsection{$m_{-}=2$}
Then $m_{+}=2$. By (\ref{exact cohom seq}) we have the following exact sequence
\begin{equation*}
0\longrightarrow H^1(M_{-},\mathbb{Z}_2)\xlongrightarrow{\theta_{-}}\mathbb{Z}_2\xlongrightarrow{\alpha_{-}}H^2(N,
\mathbb{Z}_2)\xlongrightarrow{i_{-}^*} \mathbb{Z}_2\xlongrightarrow{\theta_{-}}H^1(M_{+},\mathbb{Z}_2)\longrightarrow 0,
\end{equation*}
which gives
\begin{equation}\label{b2 equations}
\begin{array}{ll}
\mathbb{Z}_2=H^1(M_{-},\mathbb{Z}_2)\oplus \img(\alpha_{-}), &\mathbb{Z}_2=H^1(M_{+},\mathbb{Z}_2)\oplus \img(i_{-}^*), \\
H^2(N, \mathbb{Z}_2)=\img(\alpha_{-})\oplus \img(i_{-}^*).
\end{array}
\end{equation}
These equations show that $\dim H^2(N, \mathbb{Z}_2)=b_2\leq2$ and hence $b_2=1$ or $2$.

When $b_2=1$, $N\simeq \mathbb{CP}^2$ (cf. \cite{FQ90}) and either $\img(\alpha_{-})=0$, $\img(i_{-}^*)=\mathbb{Z}_2$ or $\img(\alpha_{-})=\mathbb{Z}_2$,
$\img(i_{-}^*)=0$, correspondingly either $H^1(M_{-},\mathbb{Z}_2)=\mathbb{Z}_2$, $H^1(M_{+},\mathbb{Z}_2)=0$ or $H^1(M_{-},\mathbb{Z}_2)=0$,
$H^1(M_{+},\mathbb{Z}_2)=\mathbb{Z}_2$. Therefore $M_{\pm}$ are $\mathbb{S}^2$ and $\mathbb{RP}^2$. It is well known (cf. \cite{St51}) that any orientable circle bundle with Euler number $m$ over $\mathbb{S}^2$ is a lens space $L(m,1)$\footnote{Here we ignore the specific orientations and identify $L(\pm m,1)$.} which is a quotient of $\sphere^3$ by some $\mathbb{Z}_m$-action. Recall (cf. \cite{Mo87}) that the orientable\footnote{Here orientability refers to that of the total space.} circle bundle over $\mathbb{RP}^2$ with Euler number $e$, denoted by $(On1\mid e)$, is two-fold covered by the orientable circle bundle over $\mathbb{S}^2$ with Euler number $2e$, denoted by $(Oo0\mid 2e)$. Moreover, $(On1\mid 1)\cong (Oo0\mid 4)\cong L(4,1)$. Thus $(On1\mid e)$ is the quotient of $\sphere^3$ (viewed as the unit quaternions) by the subgroup $Q_{4e}$ generated by $\omega=\cos(\pi/|e|)+i\sin(\pi/|e|)$ and $j$ (cf. \cite{La85}). Note that $Q_{4e}$ is a binary dihedral group which is abelian only if $e=1$.  It follows that the common boundary $M$, as circle bundles over
both $\mathbb{S}^2$ and $\mathbb{RP}^2$, can only be $\sphere^3/Q_4=L(4,1)$.

In conclusion, we have proved
  \begin{equation*}
  N=D(\mathbb{S}^2)\bigcup_fD(\mathbb{RP}^2),
  \end{equation*}
  where $f\in \Diff(L(4,1))$ is a diffeomorphism of the common boundary $M\cong L(4,1)$.
  On the other hand, it is known that the other (homogeneous) isoparametric foliation on $\mathbb{CP}^2$ decomposes it as (\cite{GTY11},
  \cite{Tak73})
\begin{equation*}
\mathbb{CP}^2=D(\mathbb{Q}^1)\bigcup_{id}D(\mathbb{RP}^2),
\end{equation*}
where $\mathbb{Q}^1\cong \mathbb{S}^2$ is the standard complex quadric in $\mathbb{CP}^2$ and $id$ is the identity on $M=\partial
D(\mathbb{Q}^1)=\partial D(\mathbb{RP}^2)=L(4,1)$. Due to the celebrated work on the (generalized) Smale Conjecture of Hong, et al. \cite{HKMR12},
the inclusion of the isometry group $\Iso(L(m,q))\rightarrow \Diff(L(m,q))$ of any lens space $L(m,q)$ $(m\geq3)$ is a homotopy equivalence.
As a result, $f\in \Diff(L(4,1))$ is isotopic to the identity $id$ or an orientation-preserving isometry $f_0\in \Iso(L(4,1))=O(2)\times SO(3)$ which
preserves the fibration of $L(4,1)$ over $\mathbb{S}^2$ and hence is radially extendable to an orientation-preserving diffeomorphism $\widetilde{f_0}$ of the
corresponding disk bundle $D(\mathbb{S}^2)$. This derives
\begin{equation}\label{gluingdiffeom}
N=D(\mathbb{S}^2)\bigcup_fD(\mathbb{RP}^2)\cong D(\mathbb{S}^2)\bigcup_{id~ or ~f_0}D(\mathbb{RP}^2)\cong
D(\mathbb{Q}^1)\bigcup_{id}D(\mathbb{RP}^2)=\mathbb{CP}^2,
\end{equation}
where the second diffeomorphism with respect to $f_0$ comes from gluing $\widetilde{f_0}: D(\mathbb{S}^2)\rightarrow D(\mathbb{Q}^1)$ and the identity
$id: D(\mathbb{RP}^2)\rightarrow D(\mathbb{RP}^2)$ (cf. \cite[Thm. 2.2]{Hi76}).
\begin{rem}\label{pi0part}
In fact, here and later on, it is sufficient to use only the ``$\pi_0$-part" of the Smale conjecture: $\pi_0(\Iso(M))\rightarrow \pi_0(\Diff(M))$ is an
isomorphism induced from the natural inclusion of $\Iso(M)$ in $\Diff(M)$. This part has been confirmed for every elliptic $3$-manifold (cf. \cite{Ce68},
\cite{Mc02} and references therein). In particular, $\pi_0(\Diff(L(m,1)))\cong\pi_0(\Iso(L(m,1)))\cong C_2$ is a cyclic group of order 2 for each
$m\geq1$. Moreover, it consists of (path components of) the identity and an orientation-reversing isometry $f_1$ if $m=1$ or $2$, and an
orientation-preserving isometry $f_0$, which preserves the fibration of $L(m,1)$ over $\mathbb{S}^2$ and hence is radially extendable to an orientation-preserving
diffeomorphism $\widetilde{f_0}$ of the corresponding disk bundle $D(\mathbb{S}^2)$, if $m\geq3$.
\end{rem}

Now we turn to deal with the case when $b_2=2$, i.e., $N\simeq \mathbb{S}^2\times \mathbb{S}^2$, or $\mathbb{CP}^2\# \pm \mathbb{CP}^2$ (cf.
\cite{FQ90}). By (\ref{b2 equations}) we have $H^1(M_{\pm},\mathbb{Z}_2)=0$ and thus $M_{+}\cong M_{-}\cong \mathbb{S}^2$, $M$ is a lens space
$L(m,1)$ for some $m\geq 0$. These give a splitting structure on $N$ as
\begin{equation}\label{gluingdiskbundlesoverS2}
N=D(\mathbb{S}^2)\bigcup_fD(\mathbb{S}^2),
\end{equation}
where $f\in \Diff(L(m,1))$ is a gluing diffeomorphism on the common boundary $M\cong L(m,1)$ for some $m\geq 0$. Note that $m$ is just the self
intersection number of either $\mathbb{S}^2$ in $N$ or equivalently the Euler number of the circle bundles $L(m,1)\rightarrow \mathbb{S}^2$. The proof
of this case (and hence of the total Theorem \ref{splitting thm}) will be completed by the following.
\begin{prop}
Let $N$ be a closed   (simply connected) $4$-manifold admitting a splitting structure as $(\ref{gluingdiskbundlesoverS2})$. Then $N$ is diffeomorphic to
$\mathbb{S}^2\times \mathbb{S}^2$, or $\mathbb{CP}^2\# \pm \mathbb{CP}^2$.
\end{prop}
\begin{rem}
 The classification can be explicitly described as follows. Without loss of generality, we suppose the splitting structure
 (\ref{gluingdiskbundlesoverS2}) is given by gluing two copies of an oriented disk bundle $D_m(\mathbb{S}^2)$ of opposite orientations over
 $\mathbb{S}^2$ through a diffeomorphism $f:L(m,1)\rightarrow L(m,1)$ of the boundary with induced orientation from $D_m(\mathbb{S}^2)$ for some
 integer $m\geq0$. Then $N$ is diffeomorphic to
\begin{itemize}
 \item[(1)] $\mathbb{S}^2\times \mathbb{S}^2$  if $m>2$ is even, or $m=2$ and $[f]=[id]\in \pi_0(\Diff(\mathbb{RP}^3))\cong C_2$
     (orientation-preserving), or $m=0$ and $[f]\in \pi_0(\Iso(\mathbb{S}^1\times \mathbb{S}^2))\cong C_2\times C_2\subset
     \pi_0(\Diff(\mathbb{S}^1\times \mathbb{S}^2))\cong C_2\times C_2\times C_2$ (isotopic to an isometry);
 \item[(2)] $\mathbb{CP}^2\# \mathbb{CP}^2$ if $m=2$ and $[id]\neq[f]\in \pi_0(\Diff(\mathbb{RP}^3))\cong C_2$ (orientation-reversing), or $m=1$
     and $[id]\neq[f]\in \pi_0(\Diff(\mathbb{S}^3))\cong C_2$ (orientation-reversing);
 \item[(3)] $\mathbb{CP}^2\# -\mathbb{CP}^2$ if $m>2$ is odd, or $m=1$ and $[f]=[id]\in \pi_0(\Diff(\mathbb{S}^3))\cong C_2$
     (orientation-preserving), or $m=0$ and $[f]\in \pi_0(\Diff(\mathbb{S}^1\times \mathbb{S}^2))-\pi_0(\Iso(\mathbb{S}^1\times \mathbb{S}^2))$
     (not isotopic to an isometry).
\end{itemize}
\end{rem}
\begin{proof}
It follows from van Kampen theorem that $N$ is simply connected and thus orientable, implying the splitting structure
(\ref{gluingdiskbundlesoverS2}) has the form as described in the remark above. For $m>2$, $f\in \Diff(L(m,1))$ is isotopic to the identity $id$ or a
fiber-preserving orientation-preserving isometry $f_0$ which can be extended to an orientation-preserving diffeomorphism $\widetilde{f_0}$ of the disk
bundle $D_m(\mathbb{S}^2)$ (see Remark \ref{pi0part}), both giving rise to the same double manifold $N\cong
D_m(\mathbb{S}^2)\bigcup_{id}-D_m(\mathbb{S}^2)$ as in (\ref{gluingdiffeom}). It is easily seen from \cite{St51} that this double manifold is an oriented $\mathbb{S}^2$ bundle over $\mathbb{S}^2$ and $m\in\mathbb{Z}\cong \pi_1SO(2)$ determines a reduction of the structural group $SO(3)$ to $SO(2)$ of the $2$-sphere bundle. Note that $\pi_1SO(3)\cong\mathbb{Z}_2$, thus each even $m$ corresponds to the trivial bundle and each odd $m$ corresponds to the only nontrivial bundle, the total space of which is shown to be diffeomorphic to $\mathbb{CP}^2\# -\mathbb{CP}^2$ in \cite{St51}. Hence, $N$ is $\mathbb{S}^2\times \mathbb{S}^2$ if $m$ is even, $\mathbb{CP}^2\# -\mathbb{CP}^2$ if $m$ is odd. The same
argument holds for the cases when $m=1,2$ and $f$ is orientation-preserving, or when $m=0$ and $f$ is isotopic to an isometry, since each isometry in this
case can be radially extended to the trivial disk bundle $\mathbb{S}^2\times D^2$ (see the $\pi_0$-part or the homotopy type of $\Diff(\mathbb{S}^1\times
\mathbb{S}^2)$ in \cite{Gl62}, \cite{Ha81}).

When $m=2$ (resp. $1$) and $f$ is isotopic to an orientation-reversing isometry (all in one path component) in $\Iso(\mathbb{RP}^3)$ (resp.
$\Iso(\mathbb{S}^3)$), we would first get a homeomorphism $N\simeq \mathbb{CP}^2\# \mathbb{CP}^2$ by checking the second Betti number $b_2=2$ and
the signature $\sigma=2$ using the Novikov's additivity theorem. As there is only one path component for the gluing diffeomorphism $f$, i.e., all
gluing diffeomorphisms are isotopic to each other, using the standard gluing argument as in (\ref{gluingdiffeom}) would deduce a diffeomorphism once
we establish the same splitting structure on $\mathbb{CP}^2\# \mathbb{CP}^2$ for $m=2,1$ respectively. Therefore it is sufficient to check the existence of the
following splitting structures
\begin{equation*}
\mathbb{CP}^2\# \mathbb{CP}^2=D_2(\mathbb{S}^2)\bigcup_{id}D_2(\mathbb{S}^2)=D_1(\mathbb{S}^2)\bigcup_{id}D_1(\mathbb{S}^2),
\end{equation*}
where $id$ denotes the identity on $\partial D_2(\mathbb{S}^2)=\mathbb{RP}^3$ and $\partial D_1(\mathbb{S}^2)=\mathbb{S}^3$ respectively.
The existence can be confirmed as follows. Consider the two embeddings of $\mathbb{S}^2$ in $\mathbb{CP}^2\# \mathbb{CP}^2$ as two connected sums $\mathbb{CP}^1\#\mathbb{CP}^1$ via different embeddings of $\mathbb{CP}^1\subset
\mathbb{CP}^2$, e.g., defined by the vanishing of different coordinate functions. The neighborhoods of these two embeddings will give the splitting structure for the case $m=2$. Considering the standard embeddings of $\mathbb{CP}^1$ in
the two copies of $\mathbb{CP}^2$ in $\mathbb{CP}^2\# \mathbb{CP}^2$ will give rise to the splitting structure for the case $m=1$.

The last case left to analyze is when $m=0$ and $f$ is not isotopic to an isometry, i.e., it belongs to the four ``rotation" path components of
$\Diff(\mathbb{S}^1\times \mathbb{S}^2)$ given by compositions of isometries of $\mathbb{S}^1\times \mathbb{S}^2$ with a ``rotation" (generator) diffeomorphism (cf. \cite{Gl62})
\begin{equation}\label{tau}
\tau(z,w)=(z,\phi(z)w)
 \end{equation}
 where $\phi(z)w=(z\cdot (w_1,w_2),w_3)$ is a rotation of $w=(w_1,w_2,w_3)\in \mathbb{S}^2$ around the $w_3$-axis through the angle of
 $z\in\mathbb{S}^1$. We can finally produce only one $4$-manifold from these four isotopy classes of gluing diffeomorphisms since as mentioned before,
 all isometries of $\mathbb{S}^1\times \mathbb{S}^2$ can be radially extended to diffeomorphisms of the trivial disk bundle $D_0(\mathbb{S}^2)=\mathbb{S}^2\times D^2$ and consequently their
 compositions with $\tau$ as gluing maps give rise to the same $4$-manifold as that by $\tau$ itself. It thus suffices to prove
 \begin{equation}\label{tau diffeom}
 N=\mathbb{S}^2\times D^2\bigcup_{\tau}-(\mathbb{S}^2\times D^2)\cong \mathbb{CP}^2\# -\mathbb{CP}^2.
 \end{equation}
Now we regard $N$ as given by first cutting out the sphere-factors along their equators, then gluing the disk-factors pointwise along the north and south hemispheres, at last regluing these two
hemispheres back using the transformation map $\tau_0: \mathbb{S}^1\rightarrow SO(2)$ induced from $\tau$. Then it turns out to be an oriented $\mathbb{S}^2$ bundle over
$\mathbb{S}^2$ with the reduced structural group $SO(2)\subset SO(3)$ corresponding to the element $[\tau_0]=1\in\mathbb{Z}\cong \pi_1SO(2)$. Hence $N\cong \mathbb{CP}^2\# -\mathbb{CP}^2$ as before (cf. \cite{St51}).

The proof is now complete.
\end{proof}

As a corollary, we conclude the following differential classification of singular Riemannian foliations of codimension one on closed simply connected $4$-manifolds. Note that by the result of \cite{QT13} introduced in subsection \ref{SS:SRF}, these singular Riemannian foliations can become isoparametric with only the bundle-like Riemannian metrics on the $4$-manifolds modified.
\begin{cor}\label{SRFcodim1}
Let $N$ be a closed   simply connected $4$-manifold admitting a singular Riemannian foliation $\fol$ of codimension one with regular leaf $M$ and two singular leaves $M_{\pm}$.
Then the following Table \ref{table-SRF of codim1} gives a foliated diffeomorphism classification of $(N, \fol)$ with $\fol$ in terms of $(M,M_{\pm})$:
\begin{table}[!htb]
\caption{SRF of codim 1 on closed simply connected $4$-manifolds}\label{table-SRF of codim1}
\begin{tabular}{|c|c|c|c|c|c|}
\hline
\multirow{2}{*}{$N$} & \multicolumn{2}{|c|}{$\fol$} & \multicolumn{3}{|c|}{Properties} \\
\cline{2-6}
& $M$ & $M_{\pm}$ & Homog & T-Isopar & Isopar\\
\hline
\multirow{3}{*}{$\sphere^4$}  & $L(1,1)$ &  $pt$, $pt$ &\multirow{6}{*}{Yes} & \multirow{9}{*}{Yes} & \multirow{10}{*}{Yes}\\
\cline{2-3}
& $L(0,1)$ &  $\sphere^1$, $\sphere^2$ &&&\\
\cline{2-3}
& $SO(3)/(\mathbb{Z}_2\oplus\mathbb{Z}_2)$ &  $\mathbb{RP}^2$, $\mathbb{RP}^2$ &&&\\
\cline{2-3}
\cline{1-3}
\multirow{2}{*}{$\mathbb{CP}^2$}  & $L(1,1)$ &  $pt$, $\sphere^2$ & & &\\
\cline{2-3}
& $L(4,1)$ &  $\mathbb{RP}^2$, $\sphere^2$ &&&\\
\cline{2-3}
\cline{1-3}
$\sphere^2\times\sphere^2$ & $L(2m,1)$,   $m\geq 0$ &  $\sphere^2$, $\sphere^2$ &  & &\\
\cline{1-4}
\multirow{2}{*}{$\mathbb{CP}^2\# \mathbb{CP}^2$}  & $L(1,1)$ &  $\sphere^2$, $\sphere^2$ &\multirow{2}{*}{No} &&\\
\cline{2-3}
& $L(2,1)$ &  $\sphere^2$, $\sphere^2$ &&&\\
\cline{2-3}
\cline{1-4}
\multirow{2}{*}{$\mathbb{CP}^2\# -\mathbb{CP}^2$}  & $L(2m+1,1)$,   $m\geq0$  &  $\sphere^2$, $\sphere^2$ & Yes &&\\
\cline{2-5}
& $L(0,1)$ &  $\sphere^2$, $\sphere^2$ & No & Unknown &\\
\cline{2-3}
\hline
\end{tabular}
\end{table}\\
where the column ``Homog" (resp. ``T-Isopar", ``Isopar") means whether there exist a homogeneous (resp. totally isoparametric, isoparametric) representative in the foliated diffeomorphism class.
\end{cor}
\begin{proof}
The classification is clear by the proof of Theorem \ref{splitting thm} and that of \cite[Thm. 1.1]{GT12} for $\sphere^4$, where the equivalence of the foliations with the same $(M,M_{\pm})$ follows from the consideration on isotopy classes of the gluing diffeomorphisms and the observation that all possible extensions of the gluing diffeomorphisms can be made radially (so to preserve the leaves). As for the homogeneity, one can compare with the classification of cohomogeneity one $4$-manifolds in \cite{Pa86}\footnote{In \cite{Pa86} the
second case in $\mathbb{CP}^2$ was missing, as remarked also in \cite{GZ11, Ho10}. Note also that it should be $\mathbb{CP}^2\#-\mathbb{CP}^2$ other than $\mathbb{CP}^2\#\mathbb{CP}^2$, as the covering space of the manifolds $43$ and $49$ in Parker's list.}. Naturally, the $7$ homogeneous foliations are totally isoparametric.
The $2$ nonhomogeneous foliation classes on $\mathbb{CP}^2\# \mathbb{CP}^2$ can also be represented by totally isoparametric foliations as follows. Fix an invariant metric $\tilde{g}_1$ for the homogeneous foliation $(\mathbb{CP}^2\#-\mathbb{CP}^2, \fol'=(L(1,1), \sphere^2,\sphere^2))$. This induces a singular Riemannian foliation representing $(\mathbb{CP}^2\#\mathbb{CP}^2, \fol_1=(L(1,1), \sphere^2,\sphere^2))$, by gluing the homogeneous foliations of the two disk bundles $D_1(\sphere^2)$ in $(\mathbb{CP}^2\#-\mathbb{CP}^2, \fol')$ through an orientation-reversing isometry $f_1\in \Iso(\sphere^3)$ (with respect to the metric induced from $\tilde{g}_1$), with the bundle-like metric $g_1$ on $(\mathbb{CP}^2\#\mathbb{CP}^2, \fol_1)$ also glued from the restrictions of $\tilde{g}_1$ to the two disk bundles $D_1(\sphere^2)$ by the isometry $f_1$. The metric $g_1$ is smooth and well-defined because the gluing map $f_1$ is an isometry with respect to $\tilde{g}_1|_{\partial D_1(\sphere^2)}$. It follows that each regular leaf $L(1,1)$ of $\fol_1$ has constant principal curvatures and thus $(\mathbb{CP}^2\#\mathbb{CP}^2, g_1, \fol_1)$ is totally isoparametric. The same argument is applicable to the case $(\mathbb{CP}^2\#\mathbb{CP}^2, \fol_2=(L(2,1), \sphere^2,\sphere^2))$ with the homogeneous foliation $(\sphere^2\times\sphere^2, \fol^{''}=(L(2,1),\sphere^2,\sphere^2))$ taking place of $(\mathbb{CP}^2\#-\mathbb{CP}^2, \fol')$. The last nonhomogeneous foliation class $(\mathbb{CP}^2\# -\mathbb{CP}^2,\fol=(L(0,1),\mathbb{S}^2, \mathbb{S}^2))$ can be represented by an isoparametric foliation, either by the result of \cite{QT13} cited above, or simply using the isoparametric foliation obtained from the pull-back of the standard isoparametric foliation on $\sphere^2$ through a Riemannian submersion $\mathbb{CP}^2\# -\mathbb{CP}^2\rightarrow \sphere^2$ with totally geodesic $\sphere^2$-fibres (cf. \cite{GT12}). We remark that one can not use directly the argument above for $(\mathbb{CP}^2\#\mathbb{CP}^2, \fol_1)$ to obtain a totally isoparametric representative in this case, since the gluing map now does not belong to the isometry group of $\sphere^1\times\sphere^2$.
\end{proof}
Note that we have found two nonhomogeneous examples of totally isoparametric foliations on $\mathbb{CP}^2\# \mathbb{CP}^2$ while such foliations were guessed to be homogeneous (cf. \cite{GTY11}). However, we do not know whether there exists any totally isoparametric representative in the last nonhomogeneous foliation class $(\mathbb{CP}^2\# -\mathbb{CP}^2,\fol=(L(0,1),\mathbb{S}^2, \mathbb{S}^2))$. Observing that now $\mathbb{S}^2$ represents the homology class $\pm(1,-1)\in H_2(\mathbb{CP}^2\# -\mathbb{CP}^2)$, we see that this question relates with the \emph{homology representation} (\emph{minimal genus problem}) (cf. \cite{La97, LL98}) with further geometric restrictions: the complement of the embedding sphere is an open disk bundle over another embedding sphere, and every tubular hypersurface of either embedding sphere has constant principal curvatures under some Riemannian metric on $\mathbb{CP}^2\# -\mathbb{CP}^2$. Note moreover that the latter curvature condition forces both embedding spheres to have constant principal curvatures which are independent of the choice of unit normal vectors (cf. \cite{GT13}). Nevertheless,
the answer to this question seems to be negative in view of the following observation motivated by the pull-back construction in the proof above.
\begin{prop}\label{non-t-isop}
There exists no totally isoparametric foliation $\fol=(\sphere^1\times\sphere^2, \sphere^2, \sphere^2)$ on $\mathbb{CP}^2\# -\mathbb{CP}^2$ that can be projected to a singular Riemannian foliation on $\sphere^2$ via a Riemannian submersion $\pi:\mathbb{CP}^2\# -\mathbb{CP}^2\rightarrow \sphere^2$ with totally geodesic fibres.
\end{prop}
\begin{proof}
Let $\pi:(N=\mathbb{CP}^2\# -\mathbb{CP}^2,g)\to (\sphere^2,g')$ be a Riemannian submersion with totally geodesic fibers, and let $\fol'=(\sphere^1,pt,pt)$ be a codimension 1 singular Riemannian foliation on $(\sphere^2, g')$. The pull-back foliation $\fol=\pi^{-1}(\fol')$, obtained by the preimages of the leaves in $\fol'$, is a singular Riemannian foliation on $(N,g)$, and suppose that any regular leaf $L=\pi^{-1}(\sphere^1)\cong \sphere^1\times \sphere^2$ of $\fol$ has constant principal curvatures in $N$.

As mentioned in Subsection \ref{SS:SRF}, the foliations $\fol,\fol'$ consist of level sets of some (non-unique) functions $F, f$ on $N,\sphere^2$, respectively, where $F=f\circ\pi$ is totally isoparametric on $(N,g)$ and $f$ is transnormal on $(\sphere^2, g')$. In fact, for our purpose it suffices to let $F$, $f$ be the distance functions to one of the singular leaves, which are well-defined and smooth on the regular parts $N_0=N- (\sphere^2\sqcup\sphere^2)=\pi^{-1}(N'_0)$, $N'_0=\sphere^2-(pt\sqcup pt)$ of the foliations $\fol, \fol'$. By the relation between the shape operator $S$ on $L$ with respect to the unit normal vector field $\xi:=\nabla F/|\nabla F|$ (normalized gradient) and the Hessian $H_F$ of $F$ on $(N,g)$, namely
\[
g(S(X),Y)=-H_F(X,Y)/|\nabla F|, \quad for~X,Y\in TL,
\]
$\fol$ is totally isoparametric if and only if $F$ is transnormal, i.e., $|\nabla F|^2=b(F)$ for some function $b:\mathbb{R}\rightarrow\mathbb{R}$, and $H_F|_{TL}$ has constant eigenvalues on any regular leaf $L$.

Let $L'\cong\sphere^1$ be any regular leaf in $(\sphere^2,g',\fol')$ and $L=\pi^{-1}(L')$. Let $e$ be the basic horizontal unit vector field on $L$ that projects to the (oriented) unit vector field on $L'$, and let $\{u,v\}$ be a local vertical orthonormal frame on $L$, such that $\{e,u,v\}$ is a local orthonormal frame on $L$. Notice that the basic horizontal vector fields $e,\xi$ are well-defined on the whole of the regular part $N_0$.
Recall that the O'Neill's integrability tensor $A$ acts skew-symmetrically on the horizontal orthonormal frame $\{e,\xi\}$ as
\[
A_e\xi=g(\nabla_e\xi, u)u+g(\nabla_e\xi, v)v,
\]
where $\nabla$ denotes the covariant derivative on $(N,g)$. This tensor measures the obstruction to integrability of the horizontal distribution.
When the fibres are totally geodesic, $A\equiv0$ if and only if the total space $N$ is locally a Riemannian product of the base manifold and the fiber (see more details and properties of the $A$-tensor in, e.g., \cite{Ge14}). In particular, in our case the $A$-tensor cannot vanish identically on $N$ since otherwise $N$ would have to split isometrically as $\sphere^2\times\sphere^2$.

Straightforward calculations show that under the orthonormal frame $\{e,u,v\}$,
\[
H_F|_{TL}=\left(\begin{array}{ccc} H_F(e,e) & |\nabla F|g(\nabla_e\xi, u) & |\nabla F|g(\nabla_e\xi, v) \\ |\nabla F|g(\nabla_e\xi, u)  & 0&0\\|\nabla F|g(\nabla_e\xi, v)&0&0
    \end{array}\right).
\]
Therefore, $H_F|_{TL}$ has constant eigenvalues if and only if $H_F(e,e)=\Delta f\circ\pi$ is constant (i.e., $f$ is isoparametric) and $|\nabla F|^2|A_e\xi|^2=|\nabla F|^2(g(\nabla_e\xi, u)^2+g(\nabla_e\xi, v)^2)$ is constant on $L$. Then $|A_e\xi|^2$ is constant on $L$ because $|\nabla F|^2=b(F)\neq0$ is constant on $L$. Noticing that $A_e\xi$ is now a global tangent vector field of constant length along each $\sphere^2$-fibre in $L$, we conclude that $A\equiv0$ on $L$ and hence on $N_0$. By continuity $A\equiv0$ on $N=\overline{N_0}$ which gives the contradiction as described in the preceding paragraph.
\end{proof}
\subsection{Further remark on the 3 non-homogeneous foliations}\label{3-non-homog}
To conclude this section, we would like to remark further on the exactly 3 non-homogeneous singular Riemannian foliations of codimension one in Table \ref{table-SRF of codim1} of Corollary \ref{SRFcodim1}.
In fact, these non-homogeneous foliations all arise by somewhat ``twisted" gluing of two copies of a homogeneous foliation on a disk bundle.

Explicitly,
$(\mathbb{CP}^2\# -\mathbb{CP}^2,\fol=(L(0,1)=\sphere^1\times\sphere^2,\mathbb{S}^2, \mathbb{S}^2))$ can be regarded as the pull-back of the homogeneous foliation of $\sphere^2$ by concentric circles, via the $\sphere^2$-bundle projection $\mathbb{CP}^2\# -\mathbb{CP}^2\rightarrow\sphere^2$ as in Proposition \ref{non-t-isop}.  This can thus also be seen as a twisted gluing of the homogeneous foliations on the two trivial $\sphere^2$-bundles over the south and north hemispheres, i.e., on the trivial disk bundle $\sphere^2\times D^2$. The twisted gluing diffeomorphism is nothing but the non-isometric ``rotation" $\tau$ in (\ref{tau}) (under appropriate choices of orientations of the trivial disk bundles and possibly some isotopy), which proves directly the diffeomorphism (\ref{tau diffeom}) and moreover, implies the ``failure" (Proposition \ref{non-t-isop}) of gluing two copies of the homogeneous foliation into a totally isoparametric foliation since $\tau$ is not isometric. This also gives an interesting example of non-homogeneous foliation $(M,\fol)$ that projects, via a foliated map $M\rightarrow N$, to a homogeneous foliation on $N$. Conversely, examples of homogeneous foliations that project to non-homogeneous foliations were rather well-known on the Hopf fibrations $\sphere^{2n+1}\rightarrow\mathbb{CP}^n$ (see for example \cite{GTY11}).

This ``twisted" gluing phenomenon occurs similarly for the other two non-homogeneous foliations on $\mathbb{CP}^2\#\mathbb{CP}^2$, but with isometric twisted (orientation-reversing) gluing diffeomorphisms instead, which ensures the success of the gluing into totally isoparametric foliations. One is glued from two copies of the (unique) homogeneous foliation on $D_1(\sphere^2)$, the disk bundle over $\sphere^2$ with Euler number 1 along the boundary $L(1,1)=\sphere^3$. The other is glued from two copies of the (unique) homogeneous foliation on $D_2(\sphere^2)$, the disk bundle over $\sphere^2$ with Euler number 2 along the boundary $L(2,1)=\mathbb{RP}^3$.

Notice that there are neither non-isometric, nor orientation-reversing isotopy classes of diffeomorphisms on the lens spaces $L(m,1)$ for $m>2$. This explains somewhat why there are only these 3 ``twisted" gluing cases and hence non-homogeneous classes.

\section{Singular Riemannian foliations of general codimension}\label{Proof of SRF thm}
In this section we first prove a conjecture of Molino for $4$-manifolds which reduces the objects to closed singular Riemannian foliations. Then we  prove Theorem \ref{SRF thm} by verifying the only remaining case of closed $2$-dimensional singular Riemannian foliations.

Let $(M,\fol)$ be a singular Riemannian foliation. Recall that $\ol{\fol}$ is defined as the
partition of $M$ given by the closures of the leaves in $\fol$. It is a transnormal system, i.e. the leaves are locally at a constant distance from each other. Moreover, Molino \cite{Mol88} proved that in the regular part of $\fol$, $\ol{\fol}$ is a singular Riemannian foliation, and he conjectured that $\ol{\fol}$
is actually a singular Riemannian foliation on the whole of $M$. In the following proposition, we show that Molino's conjecture holds for $4$-manifolds,
and therefore $\ol{\fol}$ is a closed singular Riemannian foliation.
\begin{prop}
Molino's conjecture holds for $4$-manifolds.
\end{prop}
\begin{proof}
If $\dim \fol=2,3$, it is easily seen that for every point $p\in M$, the infinitesimal foliation $(\sphere_p,\fol_p)$ either consists of points (if $p$ is regular), or it consists of one leaf, or it is a foliation of codimension one. All these foliations are polar, i.e., the leaf space $\sphere_p/\fol_p$ is isometric to a Riemannian orbifold, and this makes $\fol$ \emph{infinitesimally polar}. Molino's conjecture is known to hold for such foliations \cite{AL11} and thus one only needs to prove it for 1-dimensional singular Riemannian foliations.

We need to prove that there is a family of smooth vector fields $\{X_i\}$ such that, for each point $p\in M$, the tangent space of the leaf $\ol{L_p}$ of $\ol{\fol}$ through
$p$ is the span of the vectors $\{X_i( p)\}$. Notice that this is a local condition. Moreover, Molino himself proved that this condition is satisfied
around regular points of $\fol$, so we only have to prove that the condition holds around the singular leaves of $\fol$.

Since $\dim \fol=1$, the singular leaves are just points and in particular they are closed. Moreover, a metric ball around each singular leaf is foliated diffeomorphic to the orbit decomposition of a representation $\RR\to O(4)$. The closure of such actions is well known to be
homogeneous, and more precisely given by the action of a torus $T^2\to O(4)$. In particular, the closure of $\fol$ around the singular leaves of
$\fol$ is a singular foliation, which is what we wanted to prove.
\end{proof}

From the proposition above, if a $4$-manifold $M$ admits a singular Riemannian foliation $\fol$ then it also admits a closed singular Riemannian foliation $\ol{\fol}$. Moreover,
if $M$ is simply connected and closed, its Euler characteristic is positive and by \cite{Gh84} there is a compact leaf. In particular, the leaves of $\fol$ cannot be dense in $M$, and $\ol{\fol}$ is a nontrivial closed foliation.

In what follows, we will assume that $(M, \fol)$ is a closed, nontrivial singular Riemannian foliation on a closed simply connected $4$-manifold.
As introduced in Section \ref{introduction}, Theorem \ref{SRF thm} has been proven by Galaz-Garcia and the second named author [18, Cor. 8.6] for closed foliations of dimension $1$, by showing that such a foliation comes from a smooth effective circle action and then applying Fintushel's classification of 4-manifolds with a circle action [13, 14]. For foliations of dimension $3$ Theorem \ref{SRF thm} reduces to Theorem \ref{splitting thm}, and therefore we
are left to study the case when $\dim\fol=2$. Now the codimension of $\fol$ is $2$, and by \cite{Ly10} the leaf space $M^*$ is a simply connected orbifold, which has no boundary if and only if $\fol$ is a regular foliation. According to this criterion we prove the theorem separately in the following two subsections.
\subsection{}
If the foliation is regular, the leaf space $M^*$ is an orbifold without boundary. On the one hand we know that the projection $\pi:M\to M^*$ induces a surjection $\pi_1(M)\to \pi_1(M^*)$, and therefore $M^*$ must be topologically a $2$-sphere. On the other hand, by \cite[Cor. 5.3]{Ly10} the orbifold fundamental group of $M^*$ is also trivial, and therefore $M^*$ is an orbifold $2$-sphere $\sphere^2(p,q)$ with at most $2$ orbifold points of coprime order $p,q$. In particular, it is possible to write $M^*$ as a union of $2$ disks around the orbifold points, and the preimages of these disks decompose $M$ into a union of two disk
bundles along possibly exceptional leaves, which proves the result by Theorem \ref{splitting thm}. In this case it is actually possible to know more about which manifolds come up as follows.

\begin{prop}
If $(M,\fol)$ is a (regular) closed Riemannian foliation of dimension $2$ on a compact simply connected $4$-manifold $M$, then the leaves are the fibers of a $\sphere^2$-bundle over $\sphere^2$. In particular, $M$ is diffeomorphic to $\sphere^2\times\sphere^2$ or $\CC\PP^2\#-\CC\PP^2$.
\end{prop}
\begin{proof}
Let $M^*$ be the leaf space of $(M,\fol)$. The frame bundle $Fr(M^*)$ is a smooth manifold with an almost free $\O(2)$-action (cf. \cite[Thm. 1.23]{ALR07}), and the projection $P:Fr(M^*)\to M^*$ is a smooth map in the orbifold sense. If $EO(2)$ denotes the universal $\O(2)$-bundle, we construct the  \emph{Haefliger classifying space} $B=Fr(M^*)\times_{\O(2)}E\O(2)$ of $M^*$. The map $B\to M^*$ taking $[x,g]$ to $P(x)$ induces an isomorphism in rational cohomology and thus on rational homotopy (cf. \cite{Hae82}), and therefore $\pi_2(B)\otimes \QQ=\QQ$.

Up to homotopy, there is a fibration (see for example \cite{GaRa13})
\begin{equation}\label{fibration}
L\to \hat{M}\to B
\end{equation}
where $\hat{M}$ is homotopic to $M$ and $L\to \hat{M}$ is homotopic to the inclusion of a regular leaf. From the long exact sequence in homotopy of
the fibration \eqref{fibration} it follows that $\pi_1(L)\otimes \QQ=\QQ$ or $0$. Since $L$ is a compact (orientable) surface, it follows that $L=\sphere^2$. The only possible exceptional leaf would be $\RR\PP^2$. In such case the holonomy of $\fol$ (cf. Subsection \ref{SS:holonomy}) would act on
$\nu_p(\RR\PP^2)=\RR^2$ without fixed points except the origin, and therefore it would act by an orientation-preserving map. This implies that the normal bundle of $\RR\PP^2$ is orientable, but this would give a contradiction
since $M$ is orientable and $\RR\PP^2$ is not. In particular there cannot be exceptional leaves, $M^*$ is diffeomorphic to $\sphere^2$, and $M$ is an
$L=\sphere^2$-bundle over $M^*=\sphere^2$. Therefore $M$ is foliated diffeomorphic to $(\sphere^2\times \sphere^2, \sphere^2\times\{pt\})$, or
$\CC\PP^2\#-\CC\PP^2$ foliated by the fibers of the $2$-sphere bundle $\CC\PP^2\#-\CC\PP^2\to \sphere^2$.
\end{proof}
\subsection{}If the foliation is singular,  the leaf space $M^*$ is homeomorphic to a disk $D^2$, where the boundary points correspond exactly to the singular leaves of
$\fol$. In this case the leaves are orientable surfaces, and they admit a fibration over the singular leaves as described in \eqref{E:Inf_fibration}. In particular, unless the  the regular leaves are tori the only possible singular leaves are points. In such a case, let $p$ be a singular point and let $(\sphere_p, \fol_p)$ be the infinitesimal foliation at $p$. Since $\sphere_p=\sphere^3$ in this case, by \cite{Ra12} the regular leaf $\mathcal{L}$ of $\fol_p$ must be diffeomorphic to either $\sphere^2$ or $T^2$. The fibration \eqref{E:Inf_fibration} in this case becomes $\mathcal{L}\to L\to p$, where $L$ is a regular leaf of $\fol$, and thus $L$ must also be diffeomorphic to either $\sphere^2$ or $T^2$.

If the regular leaves of $\fol$ are diffeomorphic to $\sphere^2$ then the singular leaves must be points, and the singular stratum $\Sigma_0$ is the whole boundary of $M^*$, which is diffeomorphic to $\sphere^1$. In this case, $M^*$ can be written as a union of a ball of a point in
the interior, and a tubular neighborhood of the boundary. The preimage of this decomposition under $\pi$ gives $M$ a splitting structure as (\ref{splitting str}), and the result
follows from Theorem \ref{splitting thm}.

The only remaining case to consider is the one where the regular leaf $L$ of $\fol$ is diffeomorphic to $T^2$. In this case we have the following
structure \cite{GaRa13}:
\begin{itemize}
\item The leaf space $M^*$ is homeomorphic to a disk $D^2$ with boundary and corners. There are at least $2$ corners.
\item The leaves corresponding to the interior points of $M^*$ are diffeomorphic to $T^2$. The leaves corresponding to points in the boundary of
    $M^*$ (not corner points) are diffeomorphic to $\sphere^1$. The leaves corresponding to the corners are points.
\end{itemize}
The preimage of each component of $\partial M^*$ under $\pi$ is a singular stratum of $\fol$. Fix a regular leaf $L_0\cong T^2$, and for each edge $E_i$ of
$\partial M^*$  fix a leaf $L_i\cong \sphere^1$ in the preimage of the corresponding edge, and a fibration $\P_i:L_0\to L_i$. Let $(m_i,n_i)\in
\pi_1(L_0)=\ZZ\oplus \ZZ$, $\gcd(m_i,n_i)=1$, be a primitive generator of the kernel of ${\P_i}_*:\pi_1(L_0)\to \pi_1(L_i)$. Notice that $\{\pm (m_i,n_i)\}$
does not depend on the choice of $L_i$ and $\P_i$. Call $(m_i,n_i)$ the \emph{weight} of the edge $E_i$. These correspond precisely to the weights
defined in \cite{OR70}. In particular, the following properties hold:
\begin{prop}
\begin{enumerate}
\item \label{fundgp} For every $\gamma:[0,1]\to M^*$ such that $\gamma(0)\in E_i$, $\gamma(1)\in E_j$, and $\gamma'(0)\perp E_i$, $\gamma'(1)\perp
    E_j$, the preimage of $Im(\gamma)$ in $M$ is a lens space, with fundamental group $\ZZ/k\ZZ$, where $k=\left|\begin{array}{cc}m_i & m_j \\n_i &
    n_j\end{array}\right|$. Call any such curve a $(E_i,E_j)$-curve.
\item  \label{less4} If $\partial M^*$ consists of at most $4$ edges, then $M$ can be divided as a union of two disk bundles of rank greater than $1$.
\item \label{more4} If $\partial M^*$ consists of more than $4$ edges, then there exist edges $E_i$, $E_j$, $j\neq i\pm 1$, such that the preimage
    of every $(E_i,E_j)$-curve is diffeomorphic to $\sphere^3$.
\item \label{connsum} If $\partial M^*$ consists of more than $4$ edges, then $(M,\fol)$ can be written as a foliated connected sum $(\widetilde{M}_1,\fol_1)\#(\widetilde{M}_2,\fol_2)$, where $\fol_1$, $\fol_2$ are codimension $2$ singular Riemannian foliations by tori, whose leaf spaces $\widetilde{M}_1^*$, $\widetilde{M}_2^*$ have boundaries with a number of edges strictly lower than that of $\partial M^*$.
\end{enumerate}
\end{prop}
This proves Theorem \ref{SRF thm} for 2-dimensional foliations by induction and use of Theorem \ref{splitting thm}. Hence, the proof of the total Theorem \ref{SRF thm} will be complete.
\begin{proof}
\ref{fundgp}) Given a $(E_i,E_j)$-curve $\gamma$, let $S_{ij}$ be the preimage of $\gamma$ under $\pi$. It is clear that $\pi^{-1}(\gamma|_{(0,1)})$
is a manifold. Moreover, since $\gamma'(0)$ is perpendicular to $E_i$, the preimage of $\gamma|_{(0,\epsilon)}$, for $\epsilon>0$ small enough, is
diffeomorphic to a vector bundle over $\pi^{-1}(\gamma(0))=\sphere^1$, and therefore $\pi^{-1}(\gamma|_{[0,\epsilon)})$ is a manifold. By symmetry,
$\pi^{-1}(\gamma|_{(1-\epsilon,1]})$ is a manifold, and therefore $S_{ij}$ is a manifold as well.

Without loss of generality, we can suppose that $\gamma$ passes through $\pi(L_0)$. Consider the open cover $U_0=\pi^{-1}(\gamma[0,1))$,
$U_1=\pi^{-1}(\gamma(0,1])$ of $S_{ij}$. $U_0$ and $U_1$ retract to $\pi^{-1}(\gamma(0))$ and $\pi^{-1}(\gamma(1))$ respectively, which are both
diffeomorphic to $\sphere^1$, while $U_0\cap U_1$ retracts to $L_0$. Moreover, the inclusion maps $\iota_0:U_0\cap U_1\to U_0$ and $\iota_1:U_0\cap
U_1\to U_1$ are homotopic to the projections $\P_i: L_0\to L_i$ and $\P_j:L_0\to L_j$, respectively. The maps ${\P_i}_*,{\P_j}_*:\ZZ^2\to \ZZ$ induced
between the fundamental groups, are
\[
{\P_i}_*(x,y)=n_ix-m_iy,\qquad {\P_j}_*(x,y)=n_jx-m_jy.
\]
By van Kampen theorem, the fundamental group of $S_{ij}$ has a presentation
\[
\langle g_1,g_2|\; g_1^{n_i}=g_2^{n_j}, g_1^{m_i}=g_2^{m_j}\rangle.
\]
Given integers $p,q$ such that $m_ip+n_iq=1$, we obtain $g_1=g_1^{m_ip+n_iq}=g_2^{m_jp+n_jq}$ and therefore $\pi_1(S_{ij})$ is generated by $g_2$
only. Moreover, $1=g_1^{m_in_i-n_im_i}=g_2^{m_in_j-n_im_j}$ and thus $\pi_1(S_{ij})$ is cyclic of order $m_in_{j}-n_im_j$, as we wanted to prove.
\\

\ref{less4}) If $\partial M^*$ has at most $2$ edges, $M^*$ can be decomposed  as a union of two balls around $2$ points in $\partial M^*$, containing the
vertices. If $\partial M^*$ has three vertices, $M^*$ can be decomposed as the union of a disk around one vertex, and the disk around the opposite
edge. Finally, if $\partial M^*$ has four vertices, it is diffeomorphic (as an orbifold) to a rectangle, and $M^*$ can be decomposed as a union of the
disks around two opposite edges. Therefore, in each case the preimage of the decomposition in $M^*$ divides $M$ as a union of disk bundles, where the preimage of an edge is $\sphere^2$.

\ref{more4}) This point was proved in \cite[Thm. 5.7]{OR70}. For the sake of completeness, we exhibit the proof here. Let $r\geq 5$ be the number of
edges of $M^*$. Notice that the preimage of a $(E_i,E_{i+1})$-curve is diffeomorphic to the unit sphere $\sphere^3$ around the (0-dimensional) leaf that is the
preimage of the corner between $E_i$ and $E_{i+1}$. From
part \ref{fundgp}) it follows that
\begin{equation}\label{consecutive-edges}
\left|\begin{array}{cc}m_i & m_{i+1} \\n_i & n_{i+1}\end{array}\right|=\pm 1,\qquad i=1,\ldots r,
\end{equation}
where one should read $r+1=1$.
We now choose the generators of $\pi_1(L_0)=\ZZ^2$ so that the edges $E_1$ has weight $(0,1)$, and $E_2$ has weight $(1,0)$. From
\eqref{consecutive-edges}, the weights of $E_3$ and $E_r$ are, respectively, $(m_3, \pm 1)$, $(\pm 1, n_r)$. Notice that if $m_3=0$
(\emph{resp.} $m_3=\pm 1$) , then the preimage of a $(E_r,E_3)$-curve (\emph{resp.} a $(E_1,E_3)$-curve) is diffeomorphic to $\sphere^3$. In the same
way, if $n_r=0$, (\emph{resp. $n_r=\pm 1$}) then the preimage of a $(E_r,E_3)$-curve (\emph{resp.} a $(E_r,E_2)$-curve) is diffeomorphic to
$\sphere^3$. Moreover, for any $i$ if we have $|m_i|=|n_i|$, then $m_i, n_i=\pm 1$ since $m_i, n_i$ are coprime, and thus the preimage of any
$(E_1,E_i)$-curve or $(E_2,E_i)$-curve is $\sphere^3$.

Therefore, we can now restrict to the case $|m_3|>1=|n_3|$, $|n_r|>1=|m_r|$, and $|m_i|\neq |n_i|$ for every $i=3,\ldots r$. Let $3\leq i<r$ be the first
index such that $|m_i|\geq|n_i|+1$, $|n_{i+1}|\geq|m_{i+1}|+1$. The existence of such an $i$ is assured by (\ref{consecutive-edges}). Then
\[
|m_in_{i+1}|\geq (|n_i|+1)(|m_{i+1}|+1)= |m_{i+1}n_i|+|m_{i+1}|+|n_i|+1
\]
and therefore
\[
1=|m_in_{i+1}-m_{i+1}n_i|\geq |m_in_{i+1}|-|m_{i+1}n_i|\geq |m_{i+1}|+|n_i|+1.
\]
In particular $m_{i+1}=n_i=0$, which implies $|n_{i+1}|=|m_i|=1$, and applying the result in \ref{fundgp}) again, we obtain that the preimages of both $(E_1,E_i)$-curves and $(E_2,E_{i+1})$-curves are
diffeomorphic to $\sphere^3$.

\ref{connsum}) By \ref{more4}), the preimage $S\cong \sphere^3$ of some $(E_i,E_j)$-curve decomposes $M$ as a union of two connected $4$-manifolds $M_1$, $M_2$ along $S$. In the following we will focus on $M_1$, but everything will hold
for $M_2$ as well. Now $(M_1, \fol|_{M_1})$ is a foliated $4$-manifold with boundary $S$,  and the restriction of $\fol$ to $S$ gives a codimension $1$
B-foliation $(S,\fol|_{S})$. By \cite{GaRa13} it is foliated diffeomorphic to $(\sphere^3,\fol')$, where $\fol'$ is given by the orbit decomposition of the (unique up to conjugation) isometric $T^2$-action on the round $\sphere^3$. This action can be extended to an isometric action on the $4$-disk
$\DD^4$, and this action induces a singular Riemannian foliation $\fol_{\DD^4}$. The manifold
\[
\widetilde{M}_1:=M_1\cup_{S\cong \sphere^3} \DD^4
\]
is canonically foliated by the singular foliation $\fol_1$ that restricts to $\fol|_{M_1}$ on $M_1$ and $\fol_{\DD^4}$ on $\DD^4$. In the next section we will prove the following,
rather technical, lemma:
\begin{lem}[Gluing Lemma]\label{L:gluing}
There is a metric $\tilde{g}$ on $\widetilde{M}_1$ that is \emph{bundle-like} for $\fol_1$, i.e., such that $(\widetilde{M}_1,\tilde{g},\fol_1)$ is a singular Riemannian
foliation.
\end{lem}
This proves, in particular, that $(M,\fol)$ is foliated diffeomorphic to a foliated connected sum $(\widetilde{M}_1,\fol_1)\#(\widetilde{M}_2,\fol_2)$ of two codimension $2$ singular Riemannian foliations by tori. Moreover, it is easy to see that the leaf spaces $\widetilde{M}_i/\fol_i$ are obtained by gluing a corner $\DD^4/\fol_{\DD^4}\simeq\{(x,y)\in
\DD^2|\; x\geq 0, y\geq 0\}$ to $M_i/\fol|_{M_i}$ along the common boundary $\sphere^3/\fol'\simeq S/\fol|_S$, and in particular have strictly less boundary edges than
$M/\fol$ by \ref{more4}).

The proof of the proposition is now complete.
\end{proof}

\section{The gluing lemma}
The goal of this section is to prove Lemma \ref{L:gluing} above. Since $(S,\fol|_{S})$ is foliated diffeomorphic to $(\sphere^3,\fol')$ we will from
now on identify these two spaces, and suppose that the (foliated) boundary of $M_1$ is $(\sphere^3, g_0, \fol')$ with some metric $g_0$.

In what follows, we write $\sphere^3$ as a union of disk bundles
\begin{equation}\label{E:decomposition}
\sphere^3=\DD^2\times \sphere^1\bigcup_{\phi_0} -\DD^2\times \sphere^1,
\end{equation}
where $\phi_0\in \Diff(T^2=\partial(\DD^2\times \sphere^1))$ is the map interchanging the two factors,
such that the leaves of $\fol'$ are the concentric tori with respect to the canonical metric in $\DD^2\times \sphere^1$.

First of all, we observe the following.
\begin{lem}
There is a neighborhood $U$ of $\sphere^3$ in $M_1$ that is foliated diffeomorphic to $(\sphere^3\times [-\epsilon,0],\fol'\times \{pt\})$. Moreover,
the metric $g$ on $M_1$ induces a metric on $\sphere^3\times[-\epsilon,0]$ of the form $g_t+dt^2$, where $\mathfrak{F}_-=\{g_t\}_{t\in [-\epsilon,0]}$ is a
family of bundle-like metrics on $(\sphere^3,\fol')$.
\end{lem}
\begin{proof}
Consider the unique unit length normal vector field $X$ on $\sphere^3\In M_1$, pointing in the \emph{outward} direction of $M_1$. Since $X$ is
uniquely defined, its restriction to every leaf of $\fol'$ is basic with respect to the projection $M_1\rightarrow M_1/\fol|_{M_1}$. Since $\sphere^3$ is compact, there is a neighborhood $U$ of $\sphere^3$ in
$M_1$ such that the map $(p,t)\mapsto \exp_ptX(p)$ defines a diffeomorphism $\exp^{\perp}:\sphere^3\times [-\epsilon,0]\to U$. On the other hand, by
the equifocality of singular Riemannian foliations \cite[Thm. 1.5]{AT08} $\exp^{\perp}$ is also a foliated map because $X$ is basic. Moreover, $X$
is perpendicular to the family of submanifolds $S_t=\exp^{\perp}(\sphere^3\times \{t\})$, all diffeomorphic to $\sphere^3$ via $\exp^{\perp}$, and
therefore the tangent space of $U$ splits orthogonally as $TS_t\oplus \scal{X}$. If $dt$ denotes the (exact) $1$-form dual to $X$, then the metric $g$
splits as $g|_{TS_t}+dt^2$. By setting $g_t=(\exp^{\perp})^*(g|_{TS_t})$ we obtain the result.
\end{proof}

In the same way, there is a neighborhood of $(\sphere^3=\partial \DD^4, g_1,\fol')$ in $(\DD_4,\fol_{\DD^4})$ that is foliated diffeomorphic to
$\sphere^3\times [1,1+\epsilon]$, where the metric has the form $g_t+dt^2$ for some family of bundle-like metrics $\mathfrak{F}_+=\{g_t\}_{t\in[1,1+\epsilon]}$ on $\sphere^3$.

Suppose now that the metrics $g_0$ and $g_1$ on $\sphere^3$ can be connected by a family $\mathfrak{F}=\{g_t\}_{t\in [0,1]}$ of smoothly varying bundle-like metrics for $\fol'$, i.e., such that $(\sphere^3, g_t,\fol')$ is a singular Riemannian foliation for all $t\in [0,1]$. In this
case one can extend $\mathfrak{F}$ to a smooth family $\{g_t\}_{t\in[-\epsilon, 1+\epsilon]}$ using $\mathfrak{F}_\pm$. In particular
$(\sphere^3\times [0,1],g_t+dt^2,\fol'\times\{pt\})$ is a singular Riemannian foliation that can be glued smoothly with $(M_1,\fol|_{M_1})$  and
$(\DD^4,\fol_{\DD^4})$, giving a singular Riemannian foliation on
\[
M_1\bigcup (\sphere^3\times [0,1])\bigcup \DD^4\cong \widetilde{M}_1.
\]

We are thus left to produce a smooth family $\mathfrak{F}$ of bundle-like metrics connecting $g_0$ to $g_1$. Notice that it is
enough to produce a piecewise smooth family. The construction will proceed along the following steps:
\begin{enumerate}
\item \label{S:fol-diff}		Find an orientation-preserving foliated diffeomorphism $\ol{F}\in \Diff(\sphere^3,\fol')$ such that $\ol{F}^*g_0$ and $g_1$ have the same horizontal spaces near the singular leaves.
\item \label{S:F-isot-id}	Prove that $\ol{F}$ is foliated isotopic to the identity. If $\ol{F}_t$ is such an isotopy, then $\ol{F}^*_t g_0$ is a one parameter familty $\mathfrak{F}_1$ between $g_0$ and $\ol{F}^*g_0$.
\item \label{S:same-h-sp}	Find a one parameter family $\mathfrak{F}_2$ between $\ol{F}^*g_0$ to a metric $g'$ that has the same horizontal
    spaces as $g_1$.
\item \label{S:final-isot}	Find a one parameter family $\mathfrak{F}_3$ between $g'$ and $g_1$.
\end{enumerate}
The one parameter family we need is obtained by composing $\mathfrak{F}_1*\mathfrak{F}_2*\mathfrak{F}_3$. We prove each step in a separate lemma.

\begin{lem}[Step \ref{S:fol-diff}]\label{L:fol-diff}
There exist orientation-preserving foliated diffeomorphisms $F,G\in \Diff(\DD^2\times \sphere^1)$ with $F|_{\sphere^1}=id=G|_{\sphere^1}$, which glue to a diffeomorphism $\ol{F}:=(F,G)\in \Diff(\sphere^3,\fol')$ such that $\ol{F}^*g_0$ and $g_1$ have the same horizontal spaces around the singular leaves.
\end{lem}
\begin{proof}
Let $r:\DD^2\times \sphere^1\to [0,1]$ denote the radial function on $\DD^2\times \sphere^1$. Consider the \emph{homothetic transformation} with respect to $g_1$ (cf. \cite{Mol88}):
\begin{equation*}\label{E:homothetic-transf}
h^1_{\lambda}:\DD^2\times \sphere^1\to \DD^2\times \sphere^1,\qquad \lambda\in (0,1],
\end{equation*}
defined in such a way that if $q=\exp_p(x)$ for some $p\in \sphere^1$ and some $g_1$-horizontal vector $x$, then $h^1_{\lambda}(q)=\exp_p(\lambda x)$.
In the same way, define the homothetic transformation $h^0_{\lambda}$ with respect to $g_0$. Notice that for each $\lambda\in (0,1]$,
$f_{\lambda}=(h^1_{\lambda})^{-1}\circ h^0_{\lambda}$ is a foliated diffeomorphism of $\DD^2\times \sphere^1$ that restricts to the identity on
$\sphere^1$, $f_1=id$, and $f_{\lambda}$ converges smoothly to $\exp^{\perp}_{g_1}\circ(\exp^\perp_{g_0})^{-1}$ as $\lambda\to 0$, where
$\exp_{g_1}^\perp,\exp_{g_0}^\perp$ denote the normal exponential maps of $\{0\}\times\sphere^1\In\DD^2\times\sphere^1$, and the normal bundles of the singular leaf $L^+$ are identified via $\nu L^+\simeq T\sphere^3/T\sphere^1$. In particular, if we define
$f_0=\exp^{\perp}_{g_1}\circ(\exp^\perp_{g_0})^{-1}$, then $f_0|_{\sphere^1}=id$ and $f_0$ takes $g_0$-horizontal geodesics to $g_1$-horizontal geodesics (and thus $f_0^*g_0$ and $g_1$ have the same horizontal spaces). Let
\begin{align*}
f:\DD^2\times \sphere^1\times [0,1]&\lra \DD^2\times \sphere^1\\
(p,t)&\lmt f_t(p )
\end{align*}
Now consider a function $\varphi:[0,1]\to [0,1]$ such that $\varphi(t)=0$ in $[0,\epsilon)$, and $\varphi(t)=1$ in $(1-\epsilon,1]$, and define an
embedding
\begin{align*}
\iota:\DD^2\times \sphere^1&\lra \DD^2\times \sphere^1\times [0,1]\\
p&\lmt \big(p,\varphi(r(p ))\big)
\end{align*}
The composition $F:=f\circ\iota$ gives an orientation-preserving foliated diffeomorphism of $\DD^2\times \sphere^1$ that coincides with $f_0$ in
a neighborhood of $\sphere^1$, and is the identity next to the boundary.
If we denote by $G$ the same map on the other copy of $\DD^2\times \sphere^1$, we can glue the diffeomorphisms by $\phi_0$ in (\ref{E:decomposition}) and obtain the desired
 $\ol{F}=(F,G)\in \Diff(\sphere^3,\fol')$.
\end{proof}

\begin{lem}[Step \ref{S:F-isot-id}]\label{L:step ii}
The map $\ol{F}$ in Lemma \ref{L:fol-diff} is foliated isotopic to the identity.
\end{lem}
\begin{proof}
Consider tubular neighborhoods $U^{\pm}$ of small radius $\epsilon$ around the singular leaves $L^\pm$.
First, noticing that $\ol{F}=id$ at the singular leaves, we can regard $\ol{F}|_{U^{\pm}}$ alternatively as a new tubular neighborhood map $\ol{F}|_{U^{\pm}}: \nu L^{\pm}\rightarrow \sphere^3$, where we have identified $U^{\pm}$ and the normal bundles $\nu L^{\pm}$. Since tubular neighborhood maps are foliated isotopic to each other (cf. Theorems 5.3 and 6.5 in Chapter 4 in \cite{Hi76}), $\ol{F}|_{U^{\pm}}$ is foliated isotopic to the identity inclusion $id:\nu L^{\pm}\rightarrow \sphere^3$. Then by the isotopy extension theorem (Theorem 1.3 in Chapter 8 in \cite{Hi76}), $\ol{F}$ is foliated isotopic to a diffeomorphism that restricts to the identity on $U^{\pm}$.

Suppose now that $\ol{F}$ fixes $U^+\cup U^-$. Let $r=\dist_{L^+}$, where the distance is taken with respect to either metric, and let $L_t=r^{-1}(t)$. We can assume, up to rescaling the metric, that the two singular leaves correspond to $L^+=r^{-1}(0)$ and $L^-=r^{-1}(1)$. Set $T=L_{\epsilon}$. Using for example the normal holonomy with respect to $g_0$, we can identify any $L_r$, $r\in [\epsilon, 1-\epsilon]$ with $T$. Since $\ol{F}$ sends any leaf $L_r$ to itself, and it fixes every leaf $L_{t}$ with $t\in [0,\epsilon]\cup[1-\epsilon,1]$, we can think of $\ol{F}$ as a closed loop $\gamma:t\in [\epsilon,1-\epsilon]\mapsto \ol{F}|_{L_t}\in \Diff(L_t)=\Diff(T)$ based at the identity. In particular, the path lies in the identity component $\Diff_0(T)$. Let $(\DD^2\times \sphere^1,T)$ denote the homogeneous foliation where $T=T^2$ (seen as a Lie group) acts linearly on $\DD^2\times \sphere^1$. By fixing a foliated diffeomorphism $U^+\simeq \DD^2\times \sphere^1$, we obtain an action of $T$ on each $L_t$, $t\leq \epsilon$, and in particular on $T$ itself. We thus obtain a map $\iota:T\to \Diff_0(T)$ that is known to be a deformation retract of $\Diff_0(T)$. In particular, we can homotope the loop $\gamma(t)$ to a loop in $\iota(T)$, and this gives an isotopy between $\ol{F}$ and a diffeomorphism $\hat{F}$ such that $\hat{F}|_{L_t}\in \iota(T)$ for all $t$.

The map $\hat{\gamma}:t\in [\epsilon,1-\epsilon]\mapsto \hat{F}|_{L_t}\in \iota(T)$ is a loop based at the identity. However, since every diffeomorphism in $\iota(T)$ can be completed to a diffeomorphism of $U^+$ - and this completion can be made canonical using the foliated diffeomorphism $U^+\simeq \DD^2\times \sphere^1$- we can homotope the curve $\hat{\gamma}$ by only fixing the end point, and letting the initial point free to move within $\iota(T)$. In this way we obtain an isotopy through foliated diffeomorphisms that still fix $U^-$ but may not in general fix $U^+$. It is clear then, that by allowing such freedom in $\hat{\gamma}$ we can homotope it to the constant map $t\mapsto id\in \Diff_0(T)$, which corresponds to isotoping $\hat{F}$ (and hence $\ol{F}$) to the identity.
\end{proof}

\begin{lem}[Step \ref{S:same-h-sp}]
If $g, \tilde{g}$ are two bundle-like metrics of $(\sphere^3,\fol')$ with the same horizontal spaces near the singular leaves, there is a one-parameter
family of bundle-like metrics $g_t$ from $\tilde{g}$ to a metric $g'$ with the same horizontal spaces as $g$.
\end{lem}
\begin{proof}
Denote by $\sphere^3_{reg}$ the complement in $\sphere^3$ of the singular leaves. There are two horizontal distributions
$\Delta$, $\tilde{\Delta}$ given by the $g$- and $\tilde{g}$-orthogonal spaces to $\fol'$, which are both of dimension one in $\sphere^3_{reg}$. Notice moreover that, by assumption, $\Delta=\tilde{\Delta}$ in a
neighbourhood of the singular leaves. Choose a variation of distributions $\Delta_t$ from $\tilde{\Delta}$ to $\Delta$, in such a way that $\Delta_t$ is
always transverse to $\fol$, and $\Delta_t(p)$ is constant wherever $\Delta(p)=\tilde{\Delta}(p)$. For each point $p\in\sphere^3$, define
$\Delta_t^{\perp}$ the $\tilde{g}$-orthogonal distribution to $\Delta_t$, and the $\tilde{g}$-orthogonal projections $\pi_t:T_p\sphere^3\to
\Delta_t^{\perp}$, $\pi_h:T_p\sphere^3\to \tilde{\Delta}$. Finally, define
\[
h_t(x,y)= \tilde{g}(\pi_tx, \pi_ty)+\tilde{g}(\pi_hx,\pi_hy)
\]
Notice that $h_t$ is a metric, and it satisfies the following properties:
\begin{itemize}
\item If $\Delta_t$ varies smoothly, $h_t$ varies smoothly since it is defined in terms of functions that depend smoothly on $\Delta_t$.
\item The $h_t$-orthogonal space to $\fol'$ is $\Delta_t$. In fact, if $v\in T_pL_p$ and $x\in\Delta_t$, then $\pi_t(x)=0$, $\pi_h(v)=0$ and thus
    $h_t(v,x)=0$.
\item Wherever $\Delta_t(p)=\tilde{\Delta}(p)$, we also have $h_t(p)=\tilde{g}(p)$.
\item The transverse metric $h_t^T$ equals the transverse metric $\tilde{g}^T$.
\end{itemize}
In particular, $h_t$ is a bundle-like metric (cf. \cite{Al10}), and thus defining $g'=h_1$ completes the proof of the lemma.
\end{proof}

\begin{lem}[Step \ref{S:final-isot}]
Let $(M,\fol)$ be any singular Riemannian foliation, and let $g,g'$ be two bundle-like metrics with the same horizontal distribution. Then for
any $t$, the metric $g_t=tg+(1-t)g'$ is a bundle-like metric.
\end{lem}
\begin{proof}
Of course $g_t$ is a metric, and the $g_t$-horizontal distribution is the same as the $g$- and $g'$- horizontal distributions. In particular,
$g_t^T=t\cdot g^T+(1-t)\cdot g'^T$. On each stratum $\Sigma$, we can take a vertical vector field $X\in \mathfrak{X}(\fol)$, and since $g,g'$ are
bundle-like metrics of $(\Sigma,\fol|_{\Sigma})$, we have $L_Xg^T=L_Xg'^T=0$. Therefore
\[
L_Xg_t^T=t\cdot L_Xg^T+(1-t)\cdot L_Xg'^T=0
\]
and therefore $g_t$ is a bundle-like metric on each stratum. By \cite{Al10}, this is enough to ensure that $g_t$ is a bundle-like metric
for $\fol$.
\end{proof}


\begin{ack}
The authors are very grateful to Alexander Lytchak for kindly introducing each other, very nice suggestions and useful comments.
The authors also thank Zizhou Tang, Gudlaugur Thorbergsson, Burkhard Wilking and Wolfgang Ziller for their support and valuable discussion.
Many thanks also to Marcos M. Alexandrino and Fernando Galaz-Garcia for their interest and helpful conversation. The first author would like to thank the Alexander von Humboldt Foundation and the University of Cologne for their support and hospitality during his Humboldt postdoctoral position in 2012--2014, under the supervision of professor Gudlaugur Thorbergsson.
\end{ack}


\end{document}